\documentclass[12pt,a4paper]{amsart}
\usepackage[totalwidth=15.75cm,totalheight=22.275cm]{geometry}
\usepackage{amsmath,amsfonts,amssymb,verbatim,graphicx,float,amsthm}
\usepackage{color}

\usepackage[shortlabels]{enumitem}

\usepackage{subfig}
\numberwithin{equation}{section}

\usepackage{epigraph}
\usepackage{commath}

\newcommand{\modd}[1]{|{\rm d}#1|}
\newcommand{\Deriv}{{\rm D}}

\makeatletter
\def\swappedhead#1#2#3{%
  \thmnumber{\@upn{\the\thm@headfont #2\@ifnotempty{#1}{.~}}}%
  \thmname{#1}%
  \thmnote{ {\the\thm@notefont(#3)}}}
\makeatother

 \newtheoremstyle{changebreak}% name
   {9pt}%      Space above, empty = `usual value'
   {9pt}%      Space below
   {\itshape}% Body font
   {}%         Indent amount (empty = no indent, \parindent = para indent)
   {\bfseries}% Thm head font
   {.}%        Punctuation after thm head
   {\newline}% Space after thm head: \newline = linebreak
   {}%         Thm head spec

 \newtheoremstyle{defnbreak}
   {9pt}%      Space above, empty = `usual value'
   {9pt}%      Space below
   {\normalfont}% Body font
   {}%         Indent amount (empty = no indent, \parindent = para indent)
   {\bfseries}% Thm head font
   {.}%        Punctuation after thm head
   {\newline}% Space after thm head: \newline = linebreak
   {}%         Thm head spec

 \newtheoremstyle{claimstyle}%
   {}%             space above
   {}%             space below
   {\normalfont}%     body font
   {}%                indent
   {\itshape}%        header font
   {.}%               punctuation
   { }%               space after head
   {\thmnote{#3}}%    typeset note only.

\newenvironment{subproof}{\begin{proof}}{\end{proof}}

\theoremstyle{claimstyle}
\newtheorem*{varclaim}{}

\newenvironment{claim}[1][Claim]{\begin{varclaim}[#1]}{\end{varclaim}}
\newenvironment{remark}[1][Remark]{\begin{varclaim}[#1]}{\end{varclaim}}

\swapnumbers
\theoremstyle{changebreak}

\newtheorem{thm}{Theorem}[section]%

\newtheorem{lem}[thm]{Lemma}%
\newtheorem{cor}[thm]{Corollary}%

\newtheorem{prop}[thm]{Proposition}%
\newtheorem{obs}[thm]{Observation}%
\newtheorem{thmdef}[thm]{Theorem and Definition}%

\newtheorem{theo}[thm]{Theorem}

\newtheorem*{theoA}{Theorem A}
\newtheorem*{propB1}{B.1. Proposition}
\newtheorem*{propB2}{B.2. Proposition}
\newtheorem*{propB3}{B.3. Proposition}

\theoremstyle{defnbreak}
\newtheorem{rmk}[thm]{Remark}%
\newtheorem{defn}[thm]{Definition}%

\newcommand*{\defeq}{\mathrel{\vcenter{\baselineskip0.5ex \lineskiplimit0pt
                     \hbox{\scriptsize.}\hbox{\scriptsize.}}}%
                     =}

\newcommand{\N}{\mathbb{N}}

\newcommand{\C}{\ensuremath{\mathbb{C}}}
\newcommand{\Ch}{\hat{\mathbb{C}}}

\newcommand{\D}{\mathbb{D}}

\renewcommand{\theta}{\vartheta}
\renewcommand{\phi}{\varphi}

\newcommand{\dist}{\operatorname{dist}}
\newcommand{\diam}{\operatorname{diam}}

\newcommand{\interior}{\operatorname{int}}
\def \isnatural {\in\mathbb{N}}

\def \disc{\mathbb{D}}

\def \sphere{\widehat{\mathbb{C}}}
\newcommand{\tef}{transcendental entire function}
\newcommand\qfor{\quad\text{for }}
\newcommand{\classb}{\mathcal{B}}
\newcommand{\B}{\classb}
\makeatletter
\def\blfootnote{\xdef\@thefnmark{}\@footnotetext}
\makeatother
\begin{document}
\title[Class $\mathcal{B}$ or not class $\mathcal{B}$?]{Hyperbolic entire functions and the Eremenko-Lyubich class: \\
    Class $\mathcal{B}$ or not class $\mathcal{B}$?}
\author{Lasse Rempe-Gillen} 
\address{Dept. of Mathematical Sciences \\
	 University of Liverpool \\
   Liverpool L69 7ZL\\
   UK \\ ORCiD: 0000-0001-8032-8580}
\email{l.rempe@liverpool.ac.uk}
\author{Dave Sixsmith}
\address{Department of Mathematics and Statistics \\
	 The Open University \\
   Walton Hall\\
   Milton Keynes MK7 6AA\\
   UK \\
   ORCiD: 0000-0002-3543-6969}
\email{david.sixsmith@open.ac.uk}
\subjclass[2010]{Primary 37F10; Secondary 30D05, 30D15, 30D20, 37F15}
\thanks{The first author was supported by a Philip Leverhulme Prize. The second author was supported by Engineering and Physical Sciences Research Council grant EP/J022160/1.}

\dedicatory{To Alex Eremenko on the occasion of his 60th birthday}
%
%%%%%%%%%%%%%
%
% ABSTRACT
%
%%%%%%%%%%%%%
%
\begin{abstract}
 Hyperbolicity plays an important role in the study of dynamical systems, and is a key concept in the iteration of rational functions of one complex variable.
  Hyperbolic systems have also been considered in the study of transcendental entire functions. There does not appear to be
  an agreed definition of the concept in this context, due to complications arising from the non-compactness of the phase space. 
  
 In this article, we consider a natural definition of hyperbolicity that requires expanding properties on 
  the preimage of a punctured neighbourhood of the isolated singularity.
  We show that this definition is equivalent to another commonly used one: a {\tef} is hyperbolic if and only if its postsingular set is a compact subset of
  the Fatou set. This leads us to propose that this notion should be used as the general definition of hyperbolicity in the context of entire functions, 
  and, in particular,
  that speaking about hyperbolicity makes sense only within the \emph{Eremenko-Lyubich class} $\classb$ of {\tef}s with a bounded set of singular values.

  We also considerably strengthen a recent characterisation of the class $\classb$, by showing that functions outside of this class cannot be
    expanding with respect to a metric whose density decays at most polynomially. In particular, this implies that
    no transcendental entire function can be expanding with respect to the spherical metric. 
  Finally we give a characterisation of an analogous class of functions analytic in a hyperbolic domain.
\end{abstract}
\maketitle
%
%%%%%%%%%%%%%
%
% INTRO
%
%%%%%%%%%%%%%
%

\epigraph{Class $\B$ or not class $\B$, that is the question---\\
  Whether 'tis nobler in the mind to suffer\\ 
   The tracts that lie over unbounded values, \\
  Or to take arms against these spots of trouble, \\
   And by assumption, ban them?}{(loosely based on \emph{Hamlet})}
\section{Introduction}

It is a general principle in the investigation of dynamical systems that \emph{hyperbolic systems}
   (also known as ``Axiom A'', following Smale \cite{smaledynamics}) are the first class to investigate in a given setting: they exhibit the simplest behaviour, yet their study frequently leads to a better understanding of more complicated systems.
   
   In this article, we consider (non-invertible) dynamics
   in one complex variable. For rational maps, hyperbolic behaviour was understood, in general
   terms, already by Fatou \cite[pp.\ 72--73]{fatoumemoir2}, though, of course, he did not use this terminology.

  More precisely, a rational map $f\colon\Ch\to\Ch$ is said to be \emph{hyperbolic} if one of the
   following, equivalent, conditions holds \cite[Section~9.7]{beardon} (see below for definitions):
  \begin{enumerate}[(a)]
    \item the function $f$ is \emph{expanding} with respect to a suitable conformal metric
     defined on a neighbourhood of its Julia set; 
   \item every critical value of $f$ belongs to the basin of an attracting periodic cycle;
   \item the \emph{postsingular set} is a subset of the Fatou set.
  \end{enumerate}
  Moreover \cite[Theorem~4.4]{mcmullenrenormalization} every hyperbolic rational map satisfies
  \begin{enumerate}[(d)]
    \item $f$ is \emph{stable}; in other words, any nearby rational map is topologically conjugate to $f$ on its Julia set, with the conjugacy 
     depending continuously on the perturbation.\label{item:stabledef}
  \end{enumerate}

   The famous \emph{Hyperbolicity Conjecture} asserts that condition \ref{item:stabledef} is also equivalent to
    hyperbolicity; this question essentially goes back to Fatou. See
    the final sentence of Chapitre IV in \cite[p.\ 73]{fatoumemoir2}, and also compare \cite[Section~4.1]{mcmullenrenormalization} for
    a historical discussion. 

  The iteration of transcendental entire functions $f\colon\C\to\C$ also goes back to Fatou \cite{fatou}, and
    has received considerable attention in recent years. As in the rational case, the \emph{Fatou set} $F(f)\subset\C$ 
    of such a function is defined as the 
    set of $z\in\C$ such that 
    the iterates $\{f^n\}_{n\isnatural}$ form a normal family
    in a neighbourhood of $z$. Its complement, the 
    \emph{Julia set} $J(f)\defeq \C\setminus F(f)$, is the set where the dynamics is ``chaotic''. The role played by the set of critical values in rational dynamics
    is now taken on by the set $S(f)$ of (finite) \emph{singular values}, i.e.\ the closure of the set of critical and asymptotic values of $f$
    (see Section \ref{sec:singularvalues}). 

 In the transcendental setting~-- due to the 
    effect of the non-compactness of the phase space and the essential singularity at infinity~-- it is 
    not clear how ``hyperbolicity'' should be defined.  Accordingly, there is currently no accepted
    general definition; see Appendix A for a brief historical discussion. We propose
    the following notion of ``expansion'' in this setting.
    
 \begin{defn}[Expanding entire functions]\label{defn:expanding}
   A transcendental entire function $f$ is \emph{expanding} if there exist
     a connected open set $W\subset\C$, which contains $J(f)$, and a conformal metric $\rho=\rho(z)\modd{z}$ 
    on $W$ such that:
    \begin{enumerate}[(1)]
      \item $W$ contains a punctured neighbourhood of infinity, i.e.\ 
          there exists $R>0$ such that $z\in W$ whenever $|z|>R$;
      \item $f$ is expanding with respect to the metric $\rho$, i.e.\ there exists
       $\lambda>1$ such that 
          \[ \|Df(z)\|_{\rho} \defeq |f'(z)|\cdot \frac{\rho(f(z))}{\rho(z)} \geq \lambda \]
       whenever $z,f(z)\in W$; and
      \item the metric $\rho$ is complete at infinity, i.e. $\dist_\rho(z, \infty) = \infty$ whenever $z \in W$. \label{completecond}
    \end{enumerate}
 \end{defn}

 \begin{figure}
  \subfloat[$J(f_1)$]{\includegraphics[width=.45\textwidth]{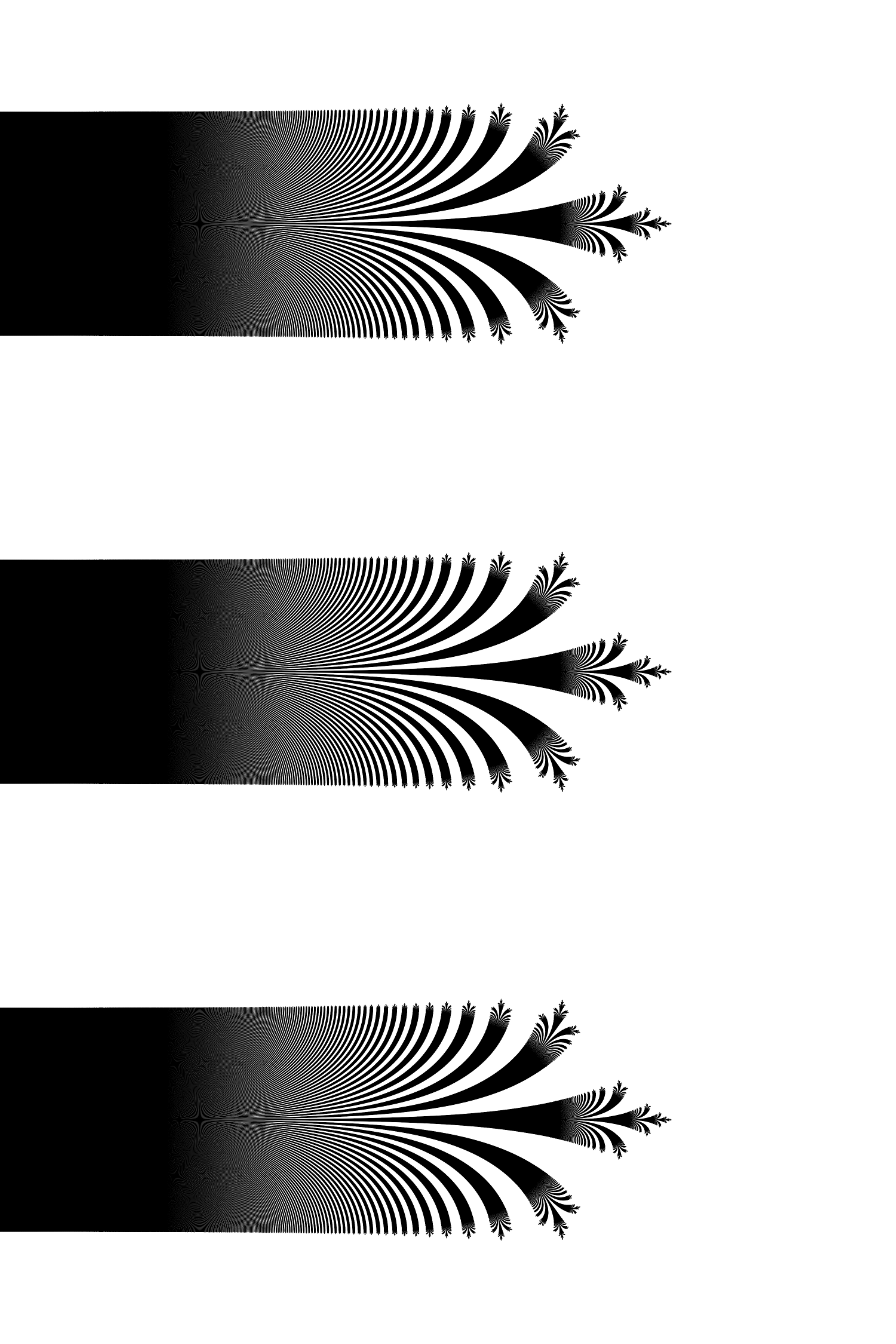}}\hfill
  \subfloat[$J(f_2)=J(f_3)$]{\includegraphics[width=.45\textwidth]{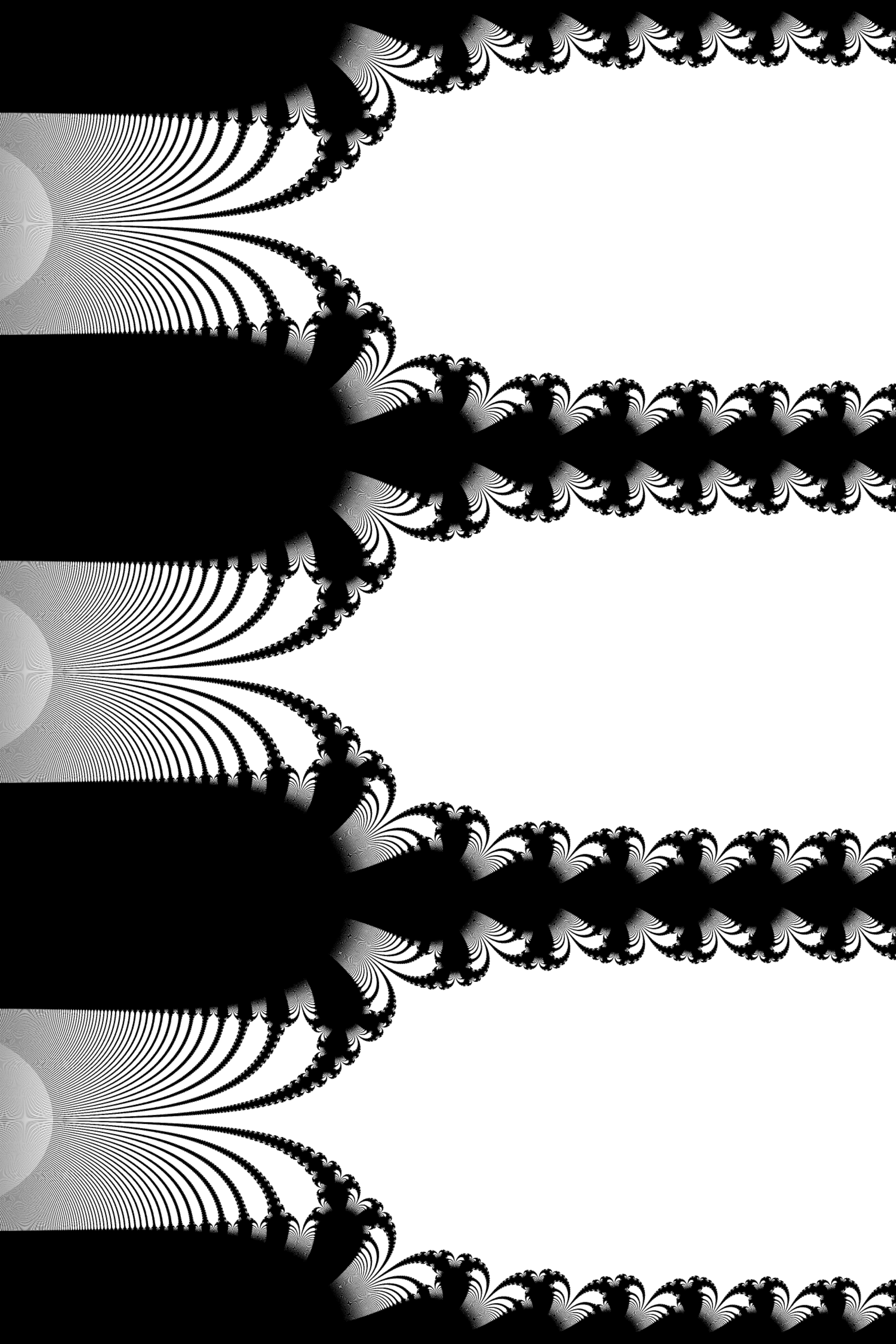}}
   \caption{\label{fig:expanding}Entire functions that are expanding on a neighbourhood of the Julia set, but 
      should not be considered hyperbolic. 
      The functions
       $f_1(z) = z+1+e^{-z}$, $f_2(z)=z-1+e^{-z}$ 
       and $f_3(z) = z-1 + 2\pi i + e^{-z}$  can be seen to be expanding on neighbourhoods (in $\C$) of
       their Julia sets (shown in black).
      $f_1$ has a Baker domain, $f_2$ has infinitely many superattracting fixed points and $f_3$ has wandering domains. None of the functions
       $f_p$ is stable under perturbations within the family $z\mapsto \lambda f_p$, $\lambda\in\C$. See Appendix~B for details, 
       and compare also Figure~\ref{fig:perturbation}.}
 \end{figure}

 We make three remarks about this definition.
 First we note that the final condition holds whenever the metric $\rho$ is complete. Indeed, as we shall see below,
 we could require that $\rho$ is complete without changing the class of expanding functions. We prefer instead to use the weaker condition \ref{completecond},
 which allows $\rho$ to be the Euclidean metric.
 
 Second, one might ask, in analogy to hyperbolicity for rational maps,
      for the metric $\rho$ to be defined and conformal on a full (rather than punctured) neighbourhood of 
     $\infty$. However, due to the nature of the essential singularity at infinity, such 
     expansion can \emph{never} be satisfied; see Corollary~\ref{cor:sphericalexpansion} below.
 
 Finally, one might require only expansion on a neighbourhood of the Julia set, as a subset of the complex plane. This is too weak
     a condition. 
  Functions with this property may have infinitely many attractors, wandering domains or \emph{Baker domains}: invariant domains of normality in which all orbits converge to infinity. (Compare Figure \ref{fig:expanding}.) 
    Such behaviour is not compatible with the usual picture of hyperbolicity and, moreover, usually not stable even under
   simple perturbations.

 Our key observation about expanding entire functions is the following: 
 \begin{thm}[Expansion only in the Eremenko-Lyubich class]\label{thm:hyperbolicity}
    If $f$ is expanding in the sense of Definition \ref{defn:expanding}, then $S(f)$ is bounded.
 \end{thm}

 In other words, every expanding function belongs to the \emph{Eremenko-Lyubich class} 
     \[ \B \defeq \{f\colon\C\to\C \text{ \normalfont transcendental entire}\colon \text{$S(f)$ \normalfont is bounded}\}. \]
  Introduced in \cite{alexmisha}, $\B$ is a large and well-studied class.
  For functions $f\in\B$ 
   there is a natural and well-established notion of hyperbolicity.
  Indeed, already McMullen~\cite[Section 6]{mcmullenarea} suggested calling an entire function 
    $f$ ``expanding'' if the postsingular set $P(f)\subset S(f)$ is a compact subset of the Fatou set
     (see condition~\ref{item:postsingularinF} in Theorem~\ref{thmdef:hyperbolicity} below). 
   The expansion property he established for such a function $f$~\cite[Proposition 6.1]{mcmullenarea} 
    is non-uniform, but
    Rippon and Stallard~\cite[Theorem~C]{ripponstallardhyperbolic} established that 
    Definition~\ref{defn:expanding} holds for $f^n$, with $\rho$ the cylinder metric, 
    and $n$ sufficiently large. In
    \cite[Lemma~5.1]{boettcher}, it is shown that $f$ itself satisfies Definition~\ref{defn:expanding} with respect
    to a
    suitable hyperbolic
    metric.

(Note  that the expanding property is stated in \cite[Theorem~C]{ripponstallardhyperbolic} 
     only for points in the Julia set, but holds also on the preimage of a 
    neighbourhood of $\infty$ by  a well-known estimate of Eremenko and Lyubich;
    see~\eqref{eqn:eremenkolyubichexpansion} below or \cite[Lemma 2.2]{ripponstallardhyperbolic}.
    We also remark that 
    Rippon and Stallard treat the more general case of transcendental \emph{meromorphic} functions.
    While we restrict here to entire functions, where Definition~\ref{defn:expanding} seems
    particularly natural, our methods apply equally to the meromorphic case; compare 
    Theorem~\ref{thm:shrinkinginfinity} and the discussion in Appendix~A.)

Rippon and Stallard \cite[p.~3253]{ripponstallardhyperbolic} 
    mention that~-- while it is not clear what the definition of a hyperbolic transcendental entire (or meromorphic)
    function should be~-- it is natural
    to call the above functions ``hyperbolic'' in view of their expansion properties. Our results suggest
    that, conversely, only these entire functions 
   should be classed as hyperbolic.

\begin{thmdef}[Hyperbolic entire functions]\label{thmdef:hyperbolicity}
  All of the following properties of a transcendental entire function $f\colon\C\to\C$ are equivalent. If any, and hence all, of these conditions hold, then $f$ is called \emph{hyperbolic}.
 \begin{enumerate}[(a)]
   \item $f$ is expanding in the sense of Definition \ref{defn:expanding};\label{item:expanding}
   \item $f$ is expanding in the sense of Definition~\ref{defn:expanding}, and  (for suitably chosen $W$) 
      the metric $\rho$ can be chosen to be the
     hyperbolic metric on $W$;\label{item:expandinghyperbolic}
   \item $f\in\B$ and every singular value belongs to the basin of an attracting periodic cycle;\label{item:singularattracting}
   \item the \emph{postsingular set}
           \[ P(f) \defeq \overline{\bigcup_{n\geq 0} f^n(S(f))} \]
          is a compact subset of $F(f)$.\label{item:postsingularinF}
 \end{enumerate}
   Moreover, each hyperbolic function is stable within its quasiconformal equivalence class.
\end{thmdef} 
\begin{remark}
  Quasiconformal equivalence classes form the natural parameter spaces of transcendental entire functions. 
    (See Proposition~\ref{prop:stability} for the formal meaning of the final statement in the theorem.) These classes were defined implicitly by 
    Eremenko and Lyubich \cite{alexmisha}, and explicitly by
    the second author~\cite{boettcher}, and provide the appropriate
    context in which to consider stability.
\end{remark}

\subsection*{Expansion near infinity} It is well-known that functions in the Eremenko-Lyubich class have strong expansion properties near
  $\infty$. Indeed, if $f\in\B$, then Eremenko and Lyubich \cite[Lemma 1]{alexmisha} proved that 
\begin{equation}\label{eqn:eremenkolyubichexpansion}
     \left|z \frac{f'(z)}{f(z)}\right| \to \infty \qquad\text{as}\qquad f(z)\to\infty. 
\end{equation}
 Observe that this quantity is precisely the derivative of $f$ with respect to the \emph{cylindrical} metric $\modd{z}/|z|$ on the punctured plane. 
  The second author \cite{sixsmithclassB} showed
  that the converse also holds: no function outside of class $\B$ exhibits this type of
   expansion near infinity. More precisely:

\begin{theoA}[Characterisation of class $\B$]
Suppose that $f$ is a {\tef}, and consider the quantity
\begin{equation}
\label{etadef}
     \eta(f) \defeq \lim_{R\rightarrow\infty}\inf_{\begin{subarray}{c} z\in\C\colon \\|f(z)|>R\end{subarray}} \left|z \frac{f'(z)}{f(z)}\right|.
\end{equation}
  Then, either $\eta(f) = \infty$ and $f\in\mathcal{B}$, or  $\eta(f) = 0$ and $f\notin\mathcal{B}$.
\end{theoA}

Theorem~\ref{thm:hyperbolicity} follows from the following strengthening of Theorem~A. 
\begin{thm}[Shrinking property near infinity]\label{thm:shrinkinginfinity}
  Let $f$ be a transcendental entire (or meromorphic) 
    function and $s\in S(f)$. Let $U$ be an open neighbourhood of $s$. Suppose that $\rho(z)\modd{z}$ is a conformal metric
    defined on $V\defeq f^{-1}(U)$ such that $\dist_{\rho}(z,\infty)=\infty$ for all $z\in V$. Then 
    \begin{equation}\label{eqn:contractionconclusion} \inf_{z\in V} \frac{|f'(z)|}{\rho(z)} = 0. \end{equation}
\end{thm}

 Note that one obtains the second half of Theorem~A by applying Theorem~\ref{thm:shrinkinginfinity} to the cylindrical metric and a sequence of 
    singular values tending to infinity. A similar application of Theorem~\ref{thm:shrinkinginfinity} gives Theorem~\ref{thm:hyperbolicity}. 
    
    Clearly some condition must be imposed upon the metric
    $\rho$ when $U$ contains no critical values. Indeed, if we let $\rho$ be the pull-back of the
    Euclidean metric under $f$, then~\eqref{eqn:contractionconclusion} fails by definition. However, it appears plausible that such a pull-back must decay 
    rapidly on some approach to infinity, and so the conclusion of Theorem~\ref{thm:shrinkinginfinity} holds for certain metrics that are not complete
    at infinity. The following theorem, another considerable strengthening of Theorem~A, shows that this is indeed the case. 

\begin{theo}[Metrics with polynomially decaying densities]
\label{thm:strong}
  Let $f$ be a transcendental entire (or meromorphic) function and $s\in S(f)$. If $U$ is an open neighbourhood of $s$, then 
 \[ \inf_{z \in f^{-1}(U)} (1+ |z|^{\tau})\cdot |f'(z)| = 0 \]
    for all $\tau>0$. 
\end{theo}

  One obtains the second half of Theorem~A from Theorem~\ref{thm:strong} by setting $\tau=1$. Similarly, one obtains Theorem~\ref{thm:shrinkinginfinity} with $\rho$ being the spherical metric by setting $\tau=2$. 
  
  Moreover, if $f$ is transcendental entire, then we can apply Theorem~\ref{thm:strong} to the function $1/f$, the singular value $s=0$ and $\tau=2$. As claimed above, this implies that the function $f$ cannot be expanding with respect to the spherical metric,
     on the preimage of a full neighbourhood of $\infty$. 

\begin{cor}[No expansion for the spherical metric]\label{cor:sphericalexpansion}
  Let $f$ be a transcendental entire function, and let $R>0$. Then
  \begin{equation}
  \label{sphexpansion}
   \inf_{\begin{subarray}{c}z\in\C\colon \\ |f(z)|>R\end{subarray}} |f'(z)|\cdot \frac{1+|z|^2}{1+|f(z)|^2} = 0.
  \end{equation} 
\end{cor}

\begin{rmk}
Corollary~\ref{cor:sphericalexpansion} can also be seen to follow from Wiman-Valiron theory, which shows that there are points of arbitrarily large modulus near which $f$ behaves like a monomial of arbitrarily large degree. It is straightforward to deduce (\ref{sphexpansion}) from this fact. 
 (More precisely, the claim follows e.g.\ from formulae (3) and (4) in the statement of the Wiman-Valiron theorem in \cite[p.~340]{alexescaping}.)
We are grateful to Alex Eremenko for this observation.
\end{rmk}

\subsection*{The Eremenko-Lyubich class on a hyperbolic domain}
Finally, we consider the analogue of the class $\mathcal{B}$ for functions analytic in a hyperbolic domain. Suppose that $\Omega\subset\C$ 
   is a hyperbolic domain
  (or, more generally, a hyperbolic Riemann surface), and that 
     $f\colon \Omega \to \mathbb{C}$ is analytic. We say that $f$ belongs to the Eremenko-Lyubich class, $\classb$, if the set of singular values of $f$ is bounded. We modify the definition (\ref{etadef}) as follows:
\begin{equation}
\label{def_of_s}
\eta_{\Omega}(f) \defeq \lim_{R\rightarrow\infty} \inf_{\begin{subarray}{c}z\in\Omega\colon \\ |f(z)|>R\end{subarray}} \| \Deriv f(z)\|,
 \end{equation}
 where the norm of the derivative is evaluated using the hyperbolic metric on $\Omega$ and the cylindrical metric on the range. 

\begin{theo}[Class $\B$ on a hyperbolic surface]
\label{thm:hyperbolicsurface}
Suppose that $\Omega$ is a hyperbolic surface and that $f\colon \Omega \to \mathbb{C}$ is analytic and unbounded. Then, either $\eta_{\Omega}(f) = \infty$ and $f\in\mathcal{B}$, or $\eta_{\Omega}(f) = 0$ and $f\notin\mathcal{B}$.
\end{theo}
\begin{rmk}
  For the second part of Theorem~\ref{thm:hyperbolicsurface}, we could replace the hyperbolic metric by any complete metric on $\Omega$, as in Theorem~\ref{thm:shrinkinginfinity}. However, Theorem~\ref{thm:hyperbolicsurface} as stated provides an appealing dichotomy in terms of the conformally
   natural quantity $\eta_{\Omega} (f)$.
\end{rmk}

\subsection*{Ideas of the proofs}
Our proof of Theorem~\ref{thm:shrinkinginfinity} considerably 
  simplifies the original proof of Theorem~A, and can be summarized as follows. The set $V$ must contain either a critical point of $f$ or 
  an asymptotic curve $\gamma$ whose image is a line segment in $U$; this is a classical and elementary fact, but its connection to
  the questions at hand appears to have been overlooked. Since the $\rho$-length of
  $\gamma$ is infinite, while the Euclidean length of $f(\gamma)$ is finite, the conclusion is immediate. 

 This argument clearly does not apply to metrics that are not complete at infinity, and hence a more detailed analysis is required 
   for the proof of Theorem~\ref{thm:strong}. Once again, we rely on elementary mapping properties of functions near a singular value, $s$, that is
   not the limit of critical values. We show that 
   there are infinitely many, pairwise disjoint, unbounded simply-connected domains on which the function in question is univalent, and which
   are mapped to round discs near $s$; see Corollary~\ref{cor:disjointtracts}. 
   Using a similar idea as in the classical proof of the Denjoy-Carleman-Ahlfors theorem, 
   it follows that, within some of these domains, the function must tend very quickly towards a corresponding asymptotic value. This leads to the 
   desired conclusion. 

  Finally, the second part of Theorem~\ref{thm:hyperbolicsurface} follows in the same manner as Theorem~\ref{thm:shrinkinginfinity}. 
    The first part, on the other hand, can be deduced in a similar way to the proof of \cite[Lemma~1]{alexmisha}), although we adopt a slightly different approach using basic properties of the hyperbolic metric. 

\subsection*{Structure of the article}
  In Section~\ref{sec:singularvalues} we give background on the notion of singular values, and prove some basic results. In Section~\ref{sec:main}, we deduce Theorems~\ref{thm:hyperbolicity}, ~\ref{thmdef:hyperbolicity} and \ref{thm:shrinkinginfinity}. 
   Sections~\ref{sec:strong} 
   and~\ref{sec:hyperbolicsurface} prove Theorems~\ref{thm:strong} and~\ref{thm:hyperbolicsurface}. Appendix A contains 
   historical remarks concerning notions of hyperbolicity for entire functions, and Appendix B concerns the functions in Figure~1.

\subsection*{Basic background and notation}
 
 We denote the complex plane, the Riemann sphere and the unit disc by $\C$, $\sphere$ and $\disc$ respectively. 
    For Euclidean discs, we use the notation
\[ B(\zeta, \ r) = \{ z \colon |z-\zeta| < r \}, \qfor 0 < r, \ \zeta\in\mathbb{C}. \]

 If $X$ is a Riemann surface, then we denote by $\infty_X$ the added point in the one-point compactification of $X$. For example, if $(z_n)$ is a sequence of points of $X$ which eventually leaves any compact subset of $X$, then $\lim_{n\rightarrow\infty} z_n = \infty_X$. 
 (Recall that, if $X$ is already compact, then $\infty_X$ is an isolated point of the  one-point compactification. This allows for
  uniformity of statements and definitions.)
 
 Closures and boundaries are always taken in an underlying Riemann surface $X$; which $X$ is meant should be clear from the context. 

 Suppose that $X$ is a Riemann surface, and $U\subset X$ is open. 
    A \emph{conformal metric} on $U$ 
    is a tensor that takes the form
    $\rho(z)\modd{z}$ in local coordinates, with $\rho$ a continuous positive function. When $X = \C$, which is the case of most interest to us,
    we can express the metric globally in this form, and we do not usually distinguish between the metric and its density function $\rho(z)$.
    
   If $\gamma \subset U$ is a locally rectifiable curve, then we denote the length of $\gamma$ with respect to the metric $\rho$ by $\ell_\rho(\gamma)$. If $z, w \in U$, then we denote the distance from $z$ to $w$ with respect to the metric $\rho$ by $\dist_\rho(z,w)$; i.e.\
    $\dist_{\rho}(z,w)=\inf_{\gamma} \ell_{\rho}(\gamma)$, where the infimum is taken over all curves connecting $z$ and $w$. Note that, by definition,
    this distance is infinite if $z$ and $w$ belong to different components of $U$. Abusing notation slightly, we define $\dist_{\rho}(z,\infty_X)\defeq \liminf_{w\to\infty_X}\dist_{\rho}(z,w)$. By definition, this quantity is infinite if 
    $U$ is relatively compact in $X$.
    
    If $z \in X$ and $S \subset X$, then we also set $\dist_\rho(z, S) \defeq \inf_{w \in S} \dist_\rho(z,w)$, and define $\diam_\rho(S) = \sup_{w_1, w_2 \in S} \dist_\rho(w_1, w_2)$.  
    When $X=\C$ and $\rho$ is the Euclidean metric, then we write simply $\ell(\gamma)$, $\dist(z,w)$ and $\dist(z, S)$. 

If $X$ is a hyperbolic surface, then we write $\rho_X$ for the hyperbolic metric in $X$ and, in local coordinates, denote its density function by $\rho_X(z)$.
 
\subsection*{Acknowledgements}
We would like to thank Walter Bergweiler and Alex Eremenko for interesting discussions about the possibility of strengthening and
   extending Theorem~A, which led us to discover the results presented in this article.

%
%%%%%%%%%%%%%%%%%%%%%%%%%%%%%%%%%%%%%%%%%%%%%%%%%%%%%%%%%%%%%%%%%%%%%%%%%%%%%%%%%%%%%%%%%%%%%%%%%%%%%%%%%%%%%%%%%%%%%%%%
%
% The singularities section
%
%%%%%%%%%%%%%%%%%%%%%%%%%%%%%%%%%%%%%%%%%%%%%%%%%%%%%%%%%%%%%%%%%%%%%%%%%%%%%%%%%%%%%%%%%%%%%%%%%%%%%%%%%%%%%%%%%%%%%%%%
%
\section{Singular values}
\label{sec:singularvalues}
In this section, we first review the definitions of singular values. While we apply them mainly for meromorphic functions defined
   on subsets of the complex plane, we introduce them in the more general setting of analytic functions between Riemann surfaces;
   see also \cite{adamthesis}. We do so to facilitate future reference, and to emphasize the general nature of our considerations. 
   In contrast to
  previous articles on similar subjects, we do not require Iversen's more precise classification of inverse function singularities, for which we refer to 
   \cite{walteralexsingularities}, \cite{walteralexcompletelyinvariant}, \cite{iversenthesis} and also \cite{eremenkosingularitiestalk}.
   Instead, we only use some basic mapping properties of functions having singular values, which we derive here from first principles.

\begin{defn}[Singular values]\label{defn:singularvalues}
  Let $X$ and $Y$ be Riemann surfaces, let $f\colon X\to Y$ be analytic, and let $s\in Y$.
  \begin{enumerate}[(a)]
    \item $s$ is called a \emph{regular value} of $f$ if there is an open neighbourhood $U$ of $s$ with the following property: if $V$ is any
       connected component of $f^{-1}(U)$, then $f\colon V\to U$ is a conformal isomorphism.
   \item $s$ is called a \emph{singular value} of $f$ if $s$ is not a regular value.
   \item $s$ is called a \emph{critical value} of $f$ if it is the image of a critical point.
   \item $s$ is called an \emph{asymptotic value} of $f$ if there is a curve $\gamma\colon [0,\infty)\to X$ such that
      $\gamma(t)\to\infty_X$ and $f(\gamma(t))\to s$ as $t\to\infty$. Such $\gamma$ is called an \emph{asymptotic curve}.\label{item:asymptoticvalue}
  \end{enumerate}
  The sets of singular, critical and asymptotic values of $f$ are denoted by $S(f)$, $CV(f)$ and $AV(f)$ respectively. 
\end{defn}
\begin{rmk}[Comments on the definition]  
  Clearly $CV(f)\cup AV(f)\subset S(f)$. On the other hand, critical and asymptotic values are dense in
   $S(f)$ (see Corollary~\ref{cor:asymptoticdense} below), so that, in fact, $S(f)=\overline{AV(f)\cup CV(f)}$.

  Equivalently to the definition above, $S(f)$ is the smallest closed subset $S$ of $Y$ such that the restriction
        $f\colon f^{-1}(Y\setminus S)\to Y\setminus S$ is a covering map. 
        
        Observe that any value in $Y\setminus \overline{f(X)}$ is a regular value of $f$. For example,
       if $X=\D$, $Y=\C$ and $f(z)\defeq z$, then the set of singular values of $f$ coincides with the unit circle
       $\partial \D$. 
\end{rmk}

The following observation shows that, near any singular value, we can find
  inverse branches defined on round discs, and containing either critical or asymptotic values
  on their boundaries.

\begin{lem}[Asymptotic values with well-behaved inverse branches]\label{lem:inversebranches}
   Let $X$ be a Riemann surface, and let $f\colon X\to\D$ be analytic, but not a conformal isomorphism. 
   Then there exists a round disc $D$ with $\overline{D}\subset \D$, a branch $\phi\colon D\to X$ of $f^{-1}$ defined on $D$, and a singular value
   $s\in \partial D$ such that either 
   \begin{enumerate}[(a)]
     \item $\lim_{z\to s} \phi(z)$ is a critical point of $f$ in $X$, and $s$ is a critical value of $f$, or \label{item:critical}
     \item $\lim_{z\to s} \phi(z)=\infty_X$. In particular, $s$ is an asymptotic value of $f$, and
       there is an asymptotic curve $\gamma$ that maps one-to-one onto
       a straight line segment $f(\gamma)$ ending at $s$.\label{item:asymptotic}
   \end{enumerate}
\end{lem}
\begin{proof}
  Fix some point $z_0\in X$ which is not a critical point of $f$. 
    By postcomposing with a M\"obius transformation, we may assume without loss of generality that
    $f(z_0)=0$. Let $\phi$ be the branch of $f^{-1}$ taking $0$ to $z_0$, and let $r>0$ be the greatest value such that $\phi$ can be continued analytically to 
    the disc $D$ of radius $r$ around $0$. 
  Then $r<1$, since $f$ is not a conformal isomorphism. It follows that there is a point $s\in \partial D\cap \D$ such that $\phi$ cannot be continued analytically
   into $s$. 

 If $\lim_{z\to s}\phi(z)=\infty_X$, then we can take $\gamma=\phi(L)$, where $L$ is the radius of $D$ ending at $s$, and the proof of case~\ref{item:asymptotic} is complete. 
  Otherwise, there is a sequence $z_n\in D$ with $z_n\to s$ and $\phi(z_n)\to c$, for some
   $c\in X$. By continuity of $f$, we have $f(c)=s$. Moreover $c$ is a critical point, since otherwise the local inverse of $f$ that maps $s$ to $c$ 
    would provide
    an analytic continuation of $\phi$ into $s$ by the identity theorem. 
\end{proof}

Discs as in Lemma~\ref{lem:inversebranches}, on which branches of the inverse are defined and which have asymptotic values on their boundary, have appeared previously
  in the study of \emph{indirect} asymptotic values in the sense of Iversen; see \cite[Proof~of~Theorem~1]{walteralexsingularities} or
  \cite[Theorem~6.2.3]{zhengvaluedistribution}. To simplify subsequent discussions, we introduce the following terminology.

\begin{defn}[Discs of univalence]
  Let $X$ and $Y$ be Riemann surfaces, and let $f\colon X\to Y$ be analytic. Suppose that $D$ is an analytic Jordan domain such that $\overline{D} \subset Y$, and that 
   $\phi\colon D\to X$ is a branch of $f^{-1}$ defined on $D$. Suppose furthermore that there is an asymptotic value $s\in\partial D$ such that 
      $\lim_{z\to s} \phi(z)=\infty_X$. 

   Then we call $D$ a \emph{disc of univalence} at the asymptotic value $s$. We also call the domain 
     $V\defeq \phi(D)\subset X$ a \emph{tract} over the disc of univalence $D$. 
\end{defn}
\begin{rmk}[Asymptotic curves with well-behaved images]\label{rmk:nicecurves}
Observe that, if $s$ is an asymptotic value for which there exists a disc of univalence, then in particular
   there is an asymptotic curve $\gamma$ for $s$ that is mapped one-to-one to an analytic arc compactly contained in $Y$,
   with one endpoint at $s$. 
\end{rmk}

The following observation is frequently useful. 
\begin{obs}[Discs of univalence for a restriction]\label{obs:discrestriction}
  Let $X$ and $Y$ be Riemann surfaces, and let $f\colon X\to Y$ be analytic. Let $U\subset Y$ be a domain, and let $W$ be a connected component of
    $f^{-1}(U)$. Then any disc of univalence of the restriction $f\colon W\to U$ at an asymptotic value $s\in U$ is also a disc of univalence of 
    $f\colon X\to Y$ at $s$. 
\end{obs}
\begin{proof}
  Let $\phi\colon D\to V \subset W$ be a branch of $f^{-1}$ as in the definition of the disc of univalence $D$. Then
    $\lim_{z\to s} \phi(z) = \infty_W$. Since $W$ was chosen to be
    a connected component of $f^{-1}(U)$, and so $f(\partial W)\subset \partial U$, 
    we deduce that $\lim_{z\to s} \phi(z) = \infty_X$.
\end{proof}

\begin{cor}[Critical and asymptotic values are dense]\label{cor:asymptoticdense}
  Let $X$ and $Y$ be Riemann surfaces, and let $f\colon X\to Y$ be analytic. Denote by $\widetilde{AV}(f)$ the set of those asymptotic values at which there exists a disc of univalence. Then $\widetilde{AV}(f)\cup CV(f)$ is dense in $S(f)$.
\end{cor}
\begin{proof}
  Let $s\in S(f)$, and let $U$ be a simply-connected neighbourhood of $s$. 
    By definition, there exists a connected component $V$ of $f^{-1}(U)$ such that $f\colon V\to U$ is not a conformal isomorphism. 
    Applying Lemma~\ref{lem:inversebranches} to $F \defeq \phi\circ (f|_V)$, where $\phi\colon U\to \D$ is a Riemann map, we find that either $U$ contains
    a critical point of $f$, or that $F$, and hence $f\colon V\to U$, has a disc of univalence. The claim follows from Observation~\ref{obs:discrestriction}.
\end{proof}

We also note the following well-known fact.

\begin{lem}[Isolated singular values]\label{lem:logarithmic}
  Let $X$ and $Y$ be Riemann surfaces, and let $f\colon X\to Y$ be analytic. Let $s$ be an isolated point of $S(f)$, let $U$ be a 
    simply-connected neighbourhood of $s$ with $U \cap CV(f) = \emptyset$, and let $V$ be a connected component of $f^{-1}(U\setminus\{s\})$. Then either 
    \begin{enumerate}[(a)]
        \item $V$ is a punctured disc, and $f\colon V\to U\setminus\{s\}$ is a finite-degree covering map, or 
        \item $V$ is simply-connected, and $f\colon V\to U\setminus\{s\}$ is a universal covering map.
    \end{enumerate}
\end{lem}
\begin{proof}
   By definition, $f\colon V\to U\setminus\{s\}$ is a covering map, and the only analytic coverings of the punctured disc are as stated
     \cite[Theorem~5.10]{forsterriemannsurfaces}.
\end{proof}

 We can now establish the following, which is crucial for the proof of Theorem~\ref{thm:strong}.

\begin{cor}[Disjoint tracts over discs of univalence]\label{cor:disjointtracts}
  Let $X$ and $Y$ be Riemann surfaces, and let $f\colon X\to Y$ be analytic, with no removable singularities at any punctures of $X$. Let $s\in S(f)$, and
   suppose that $s$ has an open neighbourhood $U$ with $U\cap CV(f)=\emptyset$. 
  Then there exist infinitely many discs of univalence $D_i\subset U$ of $f$ such that the corresponding tracts $T_i$ are pairwise disjoint. 
\end{cor}
\begin{proof}
  If $s$ is not an isolated point of $S(f)$, then the claim follows immediately from Corollary~\ref{cor:asymptoticdense}. 

  Otherwise, let $\tilde{D}$ be a round disc around $s$ in a given local chart, and apply Lemma~\ref{lem:logarithmic}. By assumption,
    $s$ is not a critical value, and $f$ does not have any removable singularities. Hence, if $V$ is any connected component of
    $f^{-1}(\tilde{D})$, then $f\colon V\to \tilde{D}\setminus s$ is a universal covering. We may thus choose 
   a single disc $D\subset\tilde{D}$ that is tangent to $s$, set 
    $D_i=D$ for all $i$, and let the $T_i$ be the infinitely many different components of $f^{-1}(D)$ in $X$. 
\end{proof}

%
%%%%%%%%%%%%%%%%%%%%%%%%%%%%%%%%%%%%%%%%%%%%%%%%%%%%%%%%%%%%%%%%%%%%%%%%%%%%%%%%%%%%%%%%%%%%%%%%%%%%%%%%%%%%%%%%%%%%%%%%
%
% The main results section
%
%%%%%%%%%%%%%%%%%%%%%%%%%%%%%%%%%%%%%%%%%%%%%%%%%%%%%%%%%%%%%%%%%%%%%%%%%%%%%%%%%%%%%%%%%%%%%%%%%%%%%%%%%%%%%%%%%%%%%%%%
%
\section{Proofs of Theorems~\ref{thm:hyperbolicity},~\ref{thmdef:hyperbolicity}~and~\ref{thm:shrinkinginfinity}}
\label{sec:main}

 The following is a general result concerning behaviour of analytic functions on sets that map close to singular values.

\begin{prop}[Analytic functions contract when mapping near singular values]\label{prop:contraction}
  Let $X$ and $Y$ be Riemann surfaces, let $f\colon X\to Y$ be analytic, let $s\in S(f)$, and let $U$ be an open neighbourhood of $s$. 
    Suppose that $V$ is a connected component of $f^{-1}(U)$ such that $f\colon V\to U$ is not a conformal isomorphism, and that 
     $\rho$ is a conformal metric on $V$ such that 
     $\dist_\rho(z,\infty_X)=\infty$ for all $z\in V$. 
     Also let $\sigma$ be any conformal metric on $U$. Then 
       \[   \inf_{z\in V} \|\Deriv f(z)\| = 0,  \]
     where the norm of the derivative is measured with respect to the metrics $\rho$ and $\sigma$.
\end{prop}
\begin{proof}
  Suppose, by way of contradiction, that $\inf_{z\in V} \|\Deriv f(z)\| > 0.$
  Then $f$ has no critical points in $V$. 
    By Corollary~\ref{cor:asymptoticdense} and Remark~\ref{rmk:nicecurves}, there is an analytic curve 
    $\gamma\subset V$ to infinity such that $\gamma$ is mapped one-to-one to an analytic arc
    $L$ compactly contained in $U$. Then
    \[ \dist_{\rho}(\gamma(0),\infty_X) \leq \ell_{\rho}(\gamma) \leq \frac{\ell_{\sigma}(L)}{\inf_{z\in\gamma} \|\Deriv f(z)\|}  
                      < \infty.\qedhere \]
\end{proof}

\begin{proof}[Proof of Theorem~\ref{thm:shrinkinginfinity}]
This follows immediately from Proposition~\ref{prop:contraction}, by taking $\sigma$ to be the Euclidean metric. 
\end{proof}    
    
\begin{remark}
  In particular, Theorem~\ref{thm:shrinkinginfinity}, remains true if we replace the full preimage $V$ by any component 
    on which $f$ is not a conformal isomorphism.
\end{remark}

\begin{proof}[Proof of Theorems~\ref{thm:hyperbolicity} and~\ref{thmdef:hyperbolicity}]
  Let $f$ be a transcendental entire function, and let $\rho$ be a conformal metric, defined on an open
   neighbourhood $W$ of $J(f)$ that contains a punctured neighbourhood of $\infty$. Suppose that $\dist_{\rho}(z,\infty)=\infty$ for all
    $z\in W$,
   and that $f$ is expanding with respect to the metric $\rho$. 
   If $c$ is a critical point of $f$, then either $c\notin W$ or $f(c)\notin W$. It follows that $W$ contains only finitely many critical values, none of which
   lie in $J(f)$. By shrinking $W$, we may hence assume that $W$ contains no critical values of $f$.

\begin{claim}[Claim 1]
  We have $W\cap S(f)=\emptyset$. In particular, $S(f)$ is bounded.
\end{claim}
\begin{subproof}
  Let $w\in W$ and let $R>0$ be such that  $A_R \defeq \{z\in\C\colon |z|>R\}\subset W$. 
    Since $w$ is not a critical value of $f$, if $D\subset W$ is a sufficiently small disc around $w$, then 
    every component of $f^{-1}(D)$ that is not contained in $A_R$  is mapped to $D$ as a conformal isomorphism. On the other hand,
    every component $V$ of $f^{-1}(D)$ that is contained in $A_R\subset W$ is also mapped as a conformal isomorphism, by 
    Proposition~\ref{prop:contraction} and the expanding property of $f$. 
\end{subproof}
  
Observe that this proves Theorem~\ref{thm:hyperbolicity}. 

\begin{claim}[Claim 2]
  We have $J(f)\cap P(f)=\emptyset$. 
\end{claim}
\begin{subproof}
  Let us set
     \[ \delta_0 \defeq \inf_{z\in J(f)} \dist_{\rho}(z,\partial W). \]
    Then $\delta_0>0$ since $W$ contains a punctured neighbourhood of $\infty$, and by assumption on $\rho$. 

   Let $w\in J(f)$, and let $\Delta_0$ be a simply-connected neighbourhood of $w$ such that $\diam_{\rho}(\Delta_0)<\delta_0$. By Claim 1, 
    $\Delta_0 \cap S(f)=\emptyset$, and hence every component $\Delta_1$ of $f^{-1}(\Delta_0)$ is mapped to $\Delta_0$ as a conformal isomorphism. 
    Furthermore, by the expanding property of $f$, we have $\diam_{\rho}(\Delta_1) < \delta_0$. Hence we can apply the preceding observation
    to $\Delta_1$, and see that any component of $f^{-2}(\Delta_0)$ is mapped univalently to $\Delta_0$.

  Proceeding inductively, we see that every branch of $f^{-n}$ can be defined on $\Delta_0$, for all $n\geq 0$. Hence $\Delta_0\cap P(f)=\emptyset$, and
   in particular $w\notin P(f)$, as required. 
\end{subproof}

Recall that every parabolic periodic cycle of a transcendental entire function, and the boundary of every Siegel disc, lies $J(f)\cap P(f)$ \cite[Theorem 7]{waltersurvey}. Furthermore, 
   the iterates in any wandering domain of $f$ converge uniformly to $(J(f)\cap P(f))\cup\{\infty\}$  \cite{limitfunctions}. Since $f\in\B$, $f$ has no Fatou component on 
   which the iterates converge uniformly to infinity \cite[Theorem~1]{alexmisha}.     So $f$ has no wandering or Baker domains.

  Thus we conclude that $F(f)$ is a union of attracting basins, and hence every point of $S(f)\subset \C\setminus W\subset F(f)$ belongs to an attracting basin.
  This completes the proof that \ref{item:expanding} implies~\ref{item:singularattracting} in Theorem~\ref{thmdef:hyperbolicity}. 

  The equivalence of~\ref{item:singularattracting} and~\ref{item:postsingularinF} is well-known; see e.g.\ \cite{ripponstallardhyperbolic} or
    \cite[Section 2]{walternurialasse}. That these
   in turn imply expansion in the sense of Definition~\ref{defn:expanding}, with $\rho$ the hyperbolic metric on a suitable domain $W$,  was shown in \cite[Lemma 5.1]{boettcher}. 
   The final claim regarding stability is made precise by Proposition~\ref{prop:stability} below.
\end{proof}

\begin{prop}[Stability of hyperbolic functions]\label{prop:stability}
Suppose that $\Lambda$ is a complex manifold, and that $(f_{\lambda})_{\lambda\in\Lambda}$ is a family of entire functions of the form
     $f_{\lambda} = \psi_{\lambda}\circ f\circ \phi_{\lambda}^{-1}$, where
     $\phi_{\lambda},\psi_{\lambda}\colon\C\to\C$ 
     are quasiconformal homeomorphisms depending analytically on the parameter $\lambda$ and $f$ is transcendental entire. 
      If $\lambda_0\in\Lambda$ is a parameter for which 
     $f_{\lambda_0}$ is hyperbolic, then $f_{\lambda_0}$ and $f_{\lambda}$ are quasiconformally conjugate on
  their Julia sets whenever $\lambda$ is sufficiently close to $\lambda_0$. Moreover, this conjugacy depends analytically on the parameter $\lambda$.
\end{prop}
\begin{proof} 
Let us first observe that hyperbolicity is an open property in any such family.
     Indeed, \ref{item:singularattracting} in Theorem~\ref{thmdef:hyperbolicity} is equivalent to the existence of a compact set $K$ with
     $f(K)\cup S(f)\subset \interior(K)$; see \cite[Proposition 2.1]{walternurialasse}. 
     Clearly, any function $\tilde{f}$ that is sufficiently close to $f$ in the sense of locally
     uniform convergence satisfies $\tilde{f}(K)~\subset~\interior(K)$. If, furthermore, $S(\tilde{f})$ is sufficiently close to $S(f)$ in the Hausdorff metric,
     then $S(\tilde{f})~\subset~\interior(K)$, and hence $\tilde{f}$ is also hyperbolic.

Therefore, if $f$ belongs to any analytic family as in the statement of the theorem, 
     then $f$ has a neighbourhood in which all maps are hyperbolic. In particular, no function in this neighbourhood has any parabolic cycles,
     and it follows that the repelling periodic points of $f$ move holomorphically over this neighbourhood. By the ``$\lambda$-lemma'' \cite{manesadsullivan},
     it follows that the closure of the set of repelling periodic points of $f$, i.e. the Julia set, also moves holomorphically. This yields the desired result.

We note that this argument uses the fact that the analytic continuation
     of a repelling periodic point encounters only algebraic singularities within a family as above. This is proved in \cite[Section~4]{alexmisha} for
     maps with finite singular sets; the same argument applies in our setting. Alternatively, the claim can also be deduced formally from the results of
     \cite{boettcher}, where it is shown that, for a given compact subset of the quasiconformal equivalence class, the set of points whose orbits remain
     sufficiently large moves holomorphically.
\end{proof}

%
%%%%%%%%%%%%%%%%%%%%%%%%%%%%%%%%%%%%%%%%%%%%%%%%%%%%%%%%%%%%%%%%%%%%%%%%%%%%%%%%%%%%%%%%%%%%%%%%%%%%%%%%%%%%%%%%%%%%%%%%
%
% The section concerning stronger contraction properties.
%
%%%%%%%%%%%%%%%%%%%%%%%%%%%%%%%%%%%%%%%%%%%%%%%%%%%%%%%%%%%%%%%%%%%%%%%%%%%%%%%%%%%%%%%%%%%%%%%%%%%%%%%%%%%%%%%%%%%%%%%%
%
\newcommand{\V}{\mathcal{V}}
\section{Metrics decaying at most polynomially}
\begin{proof}[Proof of Theorem~\ref{thm:strong}]
\label{sec:strong}
 Let $f$ be a transcendental entire or meromorphic function, let $s\in S(f)$ be a finite singular value of $f$,
 and let $U$ be an open neighbourhood of $s$.
   We may assume that there is a disc $D\subset U$ around $s$ that contains
   no critical values of $f$, as otherwise there is nothing to prove. 

 Let $K$ be any positive integer. By Corollary~\ref{cor:disjointtracts}, we can find $K$ 
  discs of univalence $D_1,\dots,D_K\subset D$, having
   pairwise disjoint tracts $G_1,\dots,G_K$. Let $a_1,\dots a_K$ be the associated asymptotic values. Also, for $1\leq n \leq K$, let
   $\Gamma_n$ be the preimage in $G_n$ of the radius of $D_n$ ending at $a_n$. 

  Then $\Gamma_n$ is an asymptotic curve for $a_n$,
   $f(\Gamma_n)$ is a straight line segment, and $f(z)\to a_n$ as $z\to\infty$ in $\Gamma_n$. This is precisely the
   setting of the proof of \cite[Theorem~1]{walteralexsingularities}, and formula (12) in that paper shows that there exist an integer $n$
   and a sequence $(w_j)$ of points tending to infinity on $\Gamma_n$ such that 
    \[ |f'(w_j)|\leq |w_j|^{-2p-1}, \]
   where $p$ is any positive integer with $4p+3<K$. (In \cite{walteralexsingularities}, the function $f$ is required to have order less
    than $p-3$, and the values $a_j$ are assumed to be pairwise distinct. However, neither of these assumptions are required for
    the proof of formula (12).) Since $K$ was arbitrary, the claim of the theorem follows. 

  The argument in \cite{walteralexsingularities} is essentially the same as in the proof of the classical Denjoy-Carleman-Ahlfors theorem: 
    since the tracts $G_n$ are unbounded and pairwise disjoint, some of them must have a small average opening angle. By the
    Ahlfors distortion theorem, it follows that $f$ must approach $a_n$ rapidly along $\Gamma_n$, which is only possible if the
    derivative becomes quite small along this curve. 
 
  For the reader's convenience, we present a self-contained proof of Theorem~\ref{thm:strong}, following the same idea. 
   Since we are not interested in precise estimates, we replace the use of the Ahlfors distortion theorem by 
   the standard estimate on the hyperbolic metric in a simply-connected domain. 
    Let $R_0>1$ be sufficiently large to ensure that each $\Gamma_n$ contains a point of modulus $R_0$. For each $n$,
    and each $z\in \Gamma_n$, let us denote by $\Gamma_n^+(z)$ the piece of $\Gamma_n$ connecting $z$ to $\infty$, 
    and the complementary bounded piece by $\Gamma_n^-(z)$.

Suppose 
    that the conclusion of the theorem did not hold. Then 
      \[ |f'(z)|\geq |z|^{-\tau}, \]
   for some $\tau>1$ and all $z\in \bigcup \Gamma_n$. This implies 
    \begin{equation}\label{eqn:approachinga}
        |f(z)-a_n| = \ell(f(\Gamma_n^+(z)))=
         \int_{\Gamma_n^+(z)} |f'(\zeta )|\modd{\zeta} \geq \int_{|z|}^{\infty} x^{-\tau}\dif x = \frac{|z|^{-(\tau-1)}}{\tau-1}.
    \end{equation} 

  We now prove that, if $K$ was chosen large enough, depending on $\tau$, then such an estimate cannot hold for all
    $\Gamma_n$. Indeed, for $x \geq R_0$, let 
   $\theta_n(x)$ denote the angular measure of the set 
      $\{\theta\colon xe^{i\theta}\in G_n\}$. Since the $G_n$ are disjoint, we have
\begin{equation}
\label{eqn:theta}
   \sum_{n=1}^K \theta_n(x) \leq 2\pi. 
\end{equation}
  We are interested in the reciprocals $1/\theta_n(x)$, since these allow us to estimate the density of the
    hyperbolic metric in $G_m$. Indeed, for $|z|\geq R_0$, it follows from the fact that $G_m$ is simply connected and \cite[Theorem I.4.3]{carlesonandgamelin} that
    \[ \rho_{G_m}(z) \geq \frac{1}{2\dist(z,\partial G_n)} \geq \frac{1}{2 |z|\theta_n}. \]
   By~\eqref{eqn:theta} and the Cauchy-Schwarz inequality, we have
    \[ \sum_{n=1}^K \frac{1}{\theta_n(x)} \geq 
       \frac{\left(\sum_{n=1}^K \frac{1}{\theta_n(x)}\right) \cdot 
            \left(\sum_{n=1}^K \theta_n(x)\right)}{2\pi} \geq
         \frac{\left(\sum_{n=1}^K 1\right)^2}{2\pi} = \frac{K^2}{2\pi}. \] 
   
  Let $x\geq R_0$ and, 
    for each $n$, choose a point $z_{n}\in \Gamma_n$ with $|z_n|=x$. Then the total hyperbolic length
   of the pieces $\Gamma_n^-(z_n)$ satisfies
   \begin{align*} \sum_{n=1}^K \ell_{G_n}(\Gamma_n^-(z_n)) &=
       \sum_{n=1}^K \int_{\Gamma_n^-} \rho_{G_n}(\zeta)\modd{\zeta} \\ &\geq
       \sum_{n=1}^K \int_{R_0}^x \frac{1}{2x\cdot \theta_n(x)} \dif x \geq 
       \int_{R_0}^x \frac{K^2}{4\pi x} \dif x = \frac{K^2}{4\pi}\cdot (\log x - \log R_0). \end{align*}
   Thus there must be a choice of $n$ and a sequence $w_j\to\infty$ in $\Gamma_n$ such that 
    \[ \ell_{G_n}(\Gamma_n^-(w_j)) \geq \frac{K}{4\pi} \log(|w_j|) + O(1) \]
     as $j\to\infty$. 
  Now $f\colon G_n\to D_n$ is a conformal isomorphism. Since $f(\Gamma_n^-(w_j))$ is a radial segment connecting the centre of $D_n$ to $f(w_j)$, we deduce that
    \[ \ell_{G_n}(\Gamma_n^-(w_j)) = \ell_{D_n}(f(\Gamma_n^-(w_j))) = \log\frac{1}{|f(w_j)-a_n|} + O(1).   \] 
   
    Hence we see that~\eqref{eqn:approachinga} cannot hold for $\tau < 1 + K/(4\pi)$. Since $K$ was arbitrary, this completes the proof. 
\end{proof}

\begin{rmk}
  We could have formulated a more general version of Theorem~\ref{thm:strong}, where $f$ is an analytic function between
   Riemann surfaces $X$ and $Y$, and $\rho$ is a conformal metric on $X$ that is complete except at finitely many
   punctures of $X$, where the metric is allowed to decay at most polynomially. 
\end{rmk}
%
%%%%%%%%%%%%%%%%%%%%%%%%%%%%%%%%%%%%%%%%%%%%%%%%%%%%%%%%%%%%%%%%%%%%%%%%%%%%%%%%%%%%%%%%%%%%%%%%%%%%%%%%%%%%%%%%%%%%%%%%
%
% The E/L class in a disc
%
%%%%%%%%%%%%%%%%%%%%%%%%%%%%%%%%%%%%%%%%%%%%%%%%%%%%%%%%%%%%%%%%%%%%%%%%%%%%%%%%%%%%%%%%%%%%%%%%%%%%%%%%%%%%%%%%%%%%%%%%
%
\section{The Eremenko-Lyubich class on a hyperbolic surface}
\label{sec:hyperbolicsurface}

\begin{proof}[Proof of Theorem~\ref{thm:hyperbolicsurface}]
  Let $\Omega$ be a hyperbolic surface, and let $f\colon \Omega\to\C$ be analytic. First suppose that
   $S(f)$ is unbounded. For every $R>0$, we may apply Proposition~\ref{prop:contraction} to $f$, taking $s$ to be 
   an element of $S(f)$ with $|s|>R$, taking $\rho$ to be the hyperbolic metric on $\Omega$, and $\sigma$ the cylindrical metric on
    $\C$. The fact that $\eta_{\Omega}(f)=0$ follows. 

  On the other hand, suppose that $S(f)$ is bounded, and let $R_0 > \max_{s\in S(f)}|s|$. Consider the domain
   $U \defeq \{z\in \C\colon |z|>R_0\}$, and its preimage $\V\defeq f^{-1}(U)$. If $V$ is a connected component of $\V$, then
   $f\colon V\to U$ is an analytic covering map, and hence a local isometry of the corresponding hyperbolic metrics. 

  By the Schwarz lemma \cite[Theorem I.4.2] {carlesonandgamelin}, we know that $\rho_V \geq \rho_\Omega$.
  On the other hand, up to a normalisation factor, the density of
   the hyperbolic metric of $U$ is given by \cite[Example 9.10]{HaymanSF2}
   \[ \rho_{U}(z) = \frac{1}{|z|\cdot (\log |z| - \log R)} = O\left(\frac{1}{|z| \log|z|}\right) \]
    as $z\to\infty$. Hence,  the derivative of $f$, measured with respect to the hyperbolic metric on $\Omega$ and the cylindrical
    metric on $\C$, satisfies
    \[ \| \Deriv f(z) \| \geq  \frac{|f'(z)|}{|f(z)|\cdot \rho_V(z)} = \frac{1}{|f(z)|\cdot\rho_U(f(z))} \to \infty \]
    as $|f(z)|\to\infty$ (where the second term should be understood in a local coordinate for $\Omega$ near $z$). 
    Hence $\eta_{\Omega}(f)=\infty$, as claimed. 
\end{proof}
%
%%%%%%%%%%%%%%%%%%%%%%%%%%%%%%%%%%%%%%%%%%%%%%%%%%%%%%%%%%%%%%%%%%%%%%%%%%%%%%%%%%%%%%%%%%%%%%%%%%%%%%%%%%%%%%%%%%%%%%%%
%
% Appendix A
%
%%%%%%%%%%%%%%%%%%%%%%%%%%%%%%%%%%%%%%%%%%%%%%%%%%%%%%%%%%%%%%%%%%%%%%%%%%%%%%%%%%%%%%%%%%%%%%%%%%%%%%%%%%%%%%%%%%%%%%%%
%
\section*{Appendix A: Definitions of hyperbolicity}
   As mentioned in the introduction, the first notion of ``expanding'' entire functions, which coincides with
    our notion of hyperbolicity in Definition~\ref{thmdef:hyperbolicity}, goes back
    to McMullen \cite{mcmullenarea}. As also mentioned, Rippon and Stallard \cite{ripponstallardhyperbolic}
    discuss hyperbolicity in the transcendental setting, arguing that functions in the class $\B$ 
    satisfying~\ref{item:singularattracting} in Theorem~\ref{thmdef:hyperbolicity} 
    deserve to be called hyperbolic.      
    However, they left open the possibility that other functions might also be classed as ``hyperbolic''.

  Mayer and Urbanski \cite{mayerurbanskitdformalism} gave a definition of hyperbolicity that relies on the Euclidean metric. 
   Specifically,   they require both uniform expansion on the Julia set and that the postsingular set be a definite (Euclidean) distance away from the Julia set.
   Without additional requirements, this includes such functions as our examples in the introduction, 
   which exhibit ``non-hyperbolic'' phenomena and are not stable under simple perturbations; see Appendix B. 

  With an additional strong regularity assumption near the Julia set, Mayer and Urbanski obtained strong and striking results 
    concerning the measurable dynamics of ``hyperbolic'' functions in the above sense; further properties of
    these functions
are described in \cite{Badenska}. We are, however, 
   not aware of any examples outside of the class $\B$ where these assumptions are known to hold. On the other hand, if $f\in\B$ is hyperbolic
   in the sense of Mayer and Urbanski, then it is also hyperbolic in the sense of Theorem~\ref{thmdef:hyperbolicity}.

 All of
   these definitions are in fact formulated, more generally, for transcendental \emph{meromorphic} functions. In this setting, there is 
   a third, and strongest, notion of hyperbolicity, studied by Zheng \cite{Zhenghyperbolic}. This requires, in addition,  that infinity is not a singular value, and
    hence can never be satisfied when $f$ is transcendental entire. 
    An example of a function with this property is $f(z) = \lambda \tan z$, for $\lambda \in (0, 1)$.
   (We note that Zheng also studied the definitions given by Rippon and Stallard and by Mayer and Urbanski; we refer to his paper for
     further details.)  

   Zheng's definition of ``hyperbolicity on the Riemann sphere'' may be considered to correspond 
    most closely to the case of hyperbolic rational functions. In particular, Zheng shows that a meromorphic function is hyperbolic in this
    sense if and only if it satisfies a certain uniform expansion property with respect to the spherical metric. 

 Our results also apply in the setting of meromorphic functions, with appropriate modifications of definitions to correctly handle prepoles. Once 
    again, they indicate that the notion of hyperbolicity does not make sense outside of the class $\B$. 

We recall also our earlier comment that, although our applications are to meromorphic functions on subsets of the complex plane, our basic results regarding the 
    properties of singularities are given for analytic functions between Riemann surfaces. A generalisation of  the study of transcendental dynamics into this 
   setting 
    is possible for the classes of ``finite type maps'', and more generally ``Ahlfors islands maps'', suggested by Epstein \cite{adamthesis,epsteinoudkerk}.
   If $W$ is a Riemann surface and $X$ is a compact Riemann surface, then an analytic function $f\colon W\to X$ is called a \emph{finite-type map} if $S(f)$ is finite and $f$ has no removable singularities at any punctures of $W$. In the case where $W\subset X$,
    Epstein develops an iteration theory that carries over the basic results from the theory of rational dynamics and of entire and meromorphic functions with finitely many
    singular values. 

  Similarly to the case of meromorphic functions, we could again consider two different notions of hyperbolicity in this setting: 
    one, analogous to that of Zheng,
    where all singular values must lie in attracting basins or map outside of the closure of $W$; and a weaker notion, in analogy to that of Rippon and Stallard, 
    where we allow singular values to lie on (or map into) the boundary of the
   domain of definition. All the standard results for the case of meromorphic functions should extend to this setting also. 

  The larger class of \emph{Ahlfors islands maps} includes all transcendental meromorphic functions, as well as all finite type maps. It is tempting to
    define an ``Eremenko-Lyubich class'' of such maps, consisting of those for which $S(f)\cap W$ is a compact subset of the domain of definition $W$. 
    Our results still apply in this setting, and imply that hyperbolicity is only to be found within this class. However, it is no longer clear that functions
    \emph{within} this class have suitable expansion properties near the boundary, and hence dichotomies such as that of Theorem A break down. 
    Finding a natural class that extends both the Eremenko-Lyubich class of entire functions and all finite type maps appears to be an interesting problem.
%
%%%%%%%%%%%%%%%%%%%%%%%%%%%%%%%%%%%%%%%%%%%%%%%%%%%%%%%%%%%%%%%%%%%%%%%%%%%%%%%%%%%%%%%%%%%%%%%%%%%%%%%%%%%%%%%%%%%%%%%%
%
% Appendix B
%
%%%%%%%%%%%%%%%%%%%%%%%%%%%%%%%%%%%%%%%%%%%%%%%%%%%%%%%%%%%%%%%%%%%%%%%%%%%%%%%%%%%%%%%%%%%%%%%%%%%%%%%%%%%%%%%%%%%%%%%%
%
\section*{Appendix B: Expansion near the Julia set for non-hyperbolic functions}

 In this section, we briefly discuss the three functions 
   \[ f_1(z) = z + 1 + e^{-z}, \quad f_2(z) = z - 1 + e^{-z} \quad\text{and}\quad f_3(z) = z - 1 + 2\pi i + e^{-z} \]
   mentioned in the introduction. All three are well-studied, and have properties that are not compatible with what would normally 
    be considered hyperbolic behaviour. 
 \begin{propB1}[Dynamical properties of the functions $f_p$]%%
 \leavevmode\vspace{-\baselineskip}%%%
 \begin{enumerate}[(a)]
   \item $f_1$ has a Baker domain containing the right half-plane $H \defeq \{ z \colon \operatorname{Re }(z) > 0\}$. That is, 
    $f_1(H)\subset H$, and $f^n(z)\to\infty$ for all $z\in H$;

   \item   $f_2$ has infinitely many superattracting fixed points, $z_n \defeq 2\pi i n$;

   \item   $f_3$ has $J(f_3)=J(f_2)$, and possesses an orbit of wandering domains. 

 \end{enumerate}
 \end{propB1}
\begin{proof}
  The function $f_1$ was first studied by Fatou. The stated property is well-known and can be verified by an elementary calculation.

  The function $f_2$ is precisely Newton's method for finding the points where $e^z=1$; it was studied in detail 
    by Weinreich \cite{weinreichthesis}. Clearly it follows directly from the definition that the points $z_n$ are
    indeed superattracting fixed points. 

   The function $f_3$ is a well-known example of a transcendental entire functions with wandering domains, first described by Herman;
      see \cite[Example~2~on p.\ 564]{bakerwanderinglms} and \cite[Section 4.5]{waltersurvey}.
The fact that $J(f_3)=J(f_2)$
    follows from the relations $f_2(z + 2\pi i) = f_2(z) + 2\pi i$ and $f_3(z) = f_2(z) + 2\pi i$. Since the points $z_n$ all belong to different
    Fatou components for $f_2$, the same is true for $f_3$. Since $f_3(z_n)= z_{n+1}$, they do indeed belong to an orbit of wandering domains 
    for $f_3$.
\end{proof}

We now justify the claim, made in the introduction, that these functions are expanding on a complex neighbourhood of the Julia set. 

\begin{propB2}[Expansion properties of the functions $f_p$]
  For each $p\in\{1,2,3\}$, there is an open neighbourhood $U$ of $J(f_p)$ in $\C$ with $f^{-1}(U)\subset U$ and a conformal metric $\rho$ on $U$ such that $f_p$ is 
    uniformly expanding
    with respect to $\rho$.
\end{propB2}
\begin{proof}
  Writing $w=-e^{-z}$, the function $f_1$ is semi-conjugate to 
    \[ F_1\colon \C\to\C; \quad w\mapsto \frac{1}{e} \cdot w \cdot e^w. \]
   This is a hyperbolic entire function, since its unique asymptotic value $0$ is an attracting fixed point, and its 
    unique critical point $-1$ belongs to the basin of attraction of this fixed point. By \cite[Lemma 5.1]{boettcher}, the function $F_1$ is expanding
    on a neighbourhood of its Julia set, with respect to a suitable hyperbolic metric. Pulling back this metric under the semiconjugacy, we obtain the 
    desired property for $f_1$.

   (We remark that, alternatively, one can show directly that $f_1$ is expanding with respect to the Euclidean metric, when restricted to a 
     suitable neighbourhood of the Julia set.) 

  The argument for $f_2$ and $f_3$ is analogous. Both functions are semi-conjugate to the map 
    \[ F_2\colon \C\to\C ;  \quad w\mapsto e\cdot w \cdot e^w.\]
     This function is not hyperbolic, as the asymptotic value $0$ is a repelling fixed point. However, we note that this asymptotic value does not correspond to any
      point in the $z$-plane under the semiconjugacy. One can hence think of $F_2$ as being hyperbolic as a self-map of 
     $\C^*=\C\setminus \{0\}$, since the critical point $-1$ is a superattracting fixed point. The same proof as in \cite{boettcher} 
    yields a neighbourhood $\tilde{U}$ of $J(F_2)\setminus\{0\}$ and a conformal metric on $U$ such that $F_2$ is expanding, and the claim follows. 
\end{proof}

 \begin{figure}
  \subfloat[$\lambda=1+0.001i$]{\includegraphics[width=.45\textwidth]{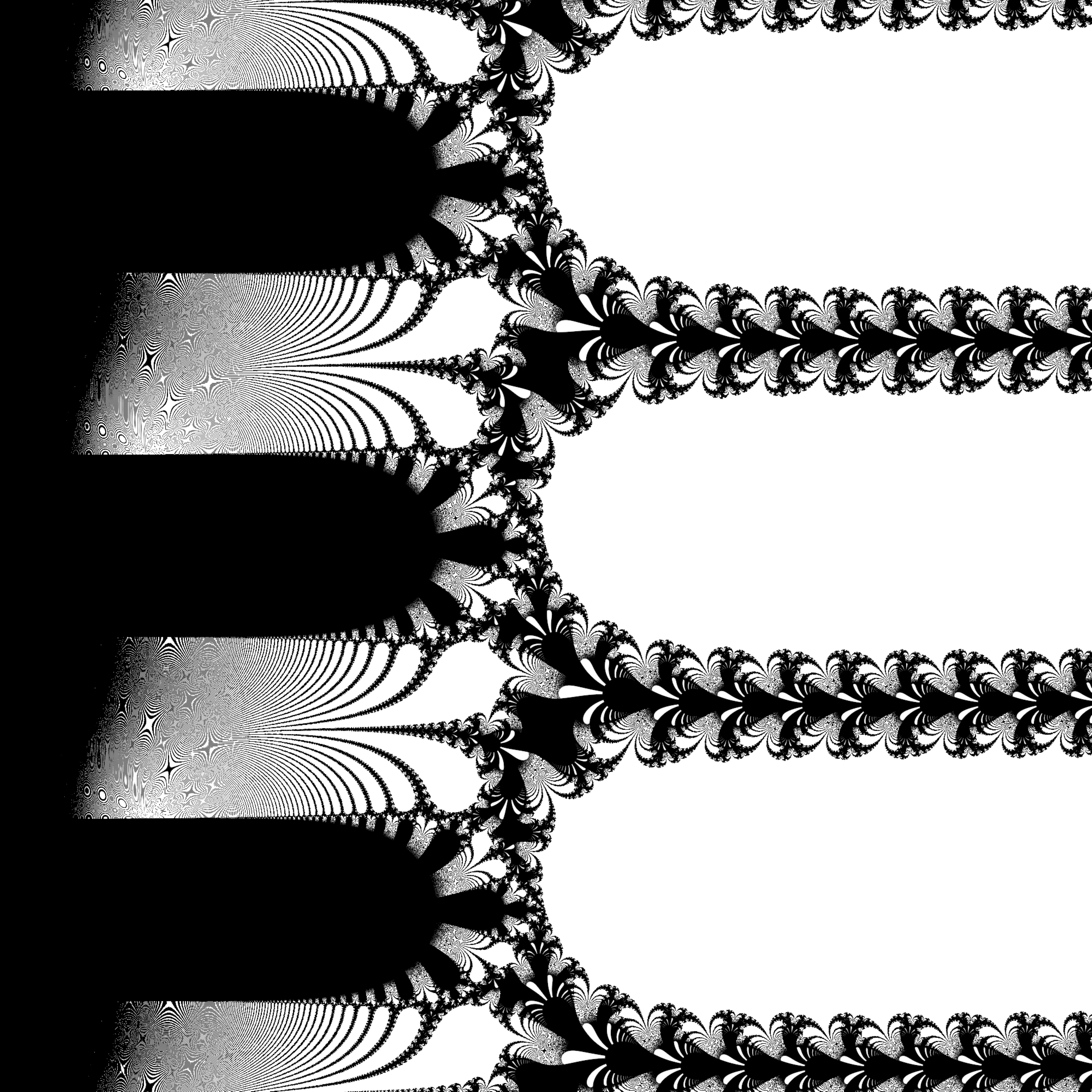}}\hfill
  \subfloat[$\lambda\approx 1.00025 + 0.00171i$]{\includegraphics[width=.45\textwidth]{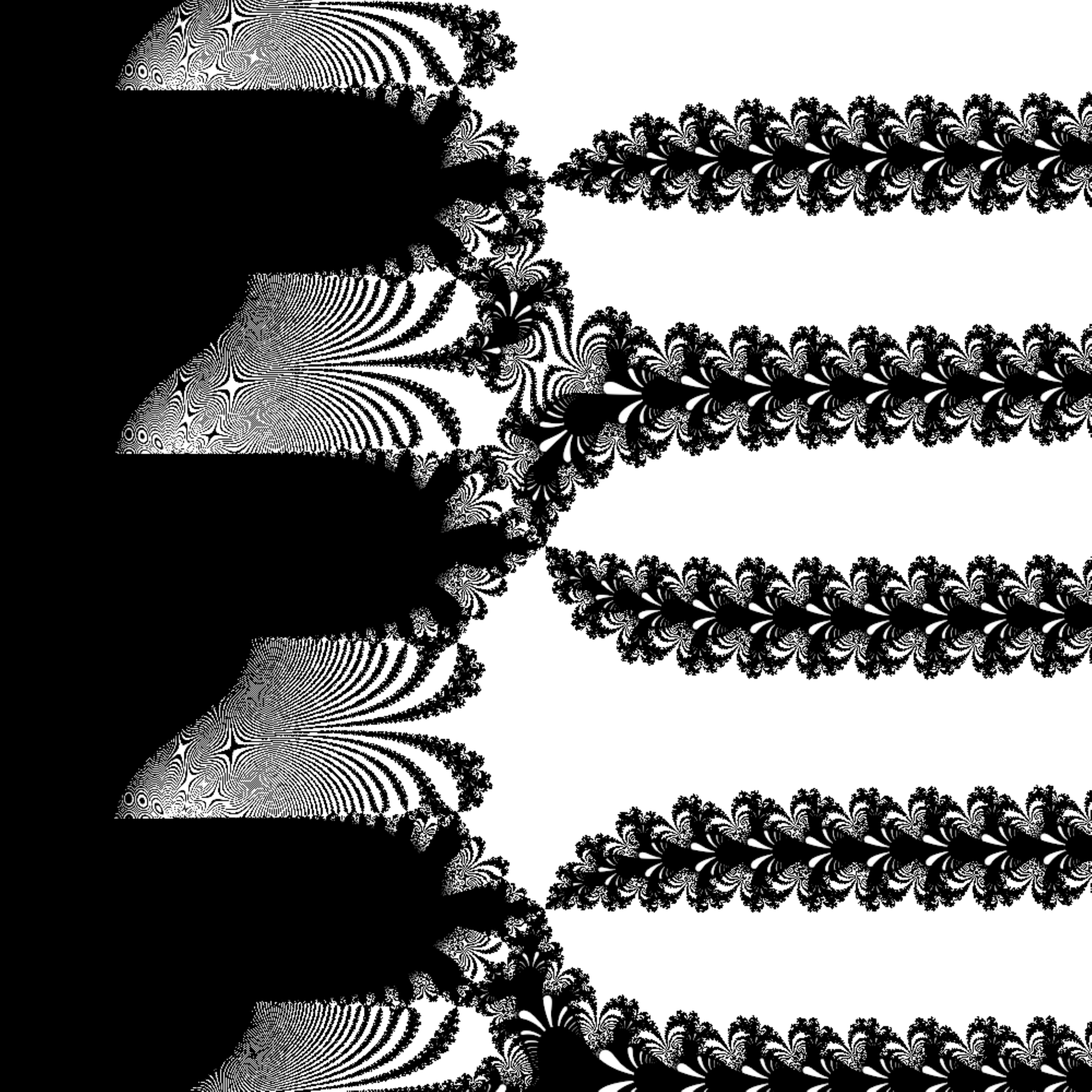}}
   \caption{\label{fig:perturbation}Instability of $f_1$, illustrating Proposition B.3. 
     Shown are the Julia sets (in black) of $J(\lambda f_1)$ for two parameter values close to $\lambda=1$, illustrating
      Proposition B.3. The second parameter is chosen so that $f_{\lambda}^2(z_{1000})=\phi(\lambda)$, as in the proof of the proposition.}
 \end{figure}

 \begin{propB3}[Instability of the functions $f_p$]
    Let $p\in \{1,2,3\}$. Then $f_p$ is not stable in the family $(\lambda f_p)_{\lambda\in\C}$;
      more precisely, there exist values of $\lambda$ arbitrarily close to $1$ such that $\lambda f_p$ and $f_p$ are not topologically conjugate on their
       Julia sets.
  \end{propB3}
  
  \begin{proof}
We prove only the case $p=1$; the proofs in the other cases are very similar. For simplicity, we write $f\defeq f_1$ and 
    $f_{\lambda} \defeq \lambda f$. 
Consider the critical points of $f_\lambda$, $z_n \defeq 2n\pi i$. Observe that $z_n \in F(f)$ for all $n$. We claim that there are values of $\lambda$ arbitrarily close to $1$ such that, for some sufficiently large value of $n$, we have $z_n \in J(f_\lambda)$. In this case $J(f)$ and $J(f_{\lambda})$ are 
 not topologically conjugate on their Julia sets, since such a conjugacy must preserve the local degree of $f$. 

In order to prove this claim, observe that we can analytically continue the repelling fixed point $i\pi$ of $f$ as a solution of the equation
     $f_{\lambda}(z)=z$ in a neighbourhood of $\lambda=1$, by the implicit function theorem. That is, if $\delta\in (0,1)$ is sufficiently small, then there is 
     an analytic function $\phi\colon B(1,\delta) \to B(0,2\pi)$ such that $\phi(\lambda)$ is a repelling fixed point of $f_{\lambda}$ for all 
     $\lambda\in B(1,\delta)$.

It can be seen that there are infinitely many zeros of $f$, and that we can number a subsequence of them, $(\xi_m)_{m\in\N}$, so that $\operatorname{Im }(\xi_m) \rightarrow \infty$ and $\operatorname{Re }(\xi_m) \sim -\log(\operatorname{Im }(\xi_m))$ as $m\rightarrow\infty$. By a further calculation, 
  there exists $r>0$ such that, for sufficiently large values of $m$,  $f$ (and hence $f_\lambda$) is univalent in $B(\xi_m, r)$. 
    Since $|f_\lambda'(\xi_m)|\rightarrow\infty$ as $n\rightarrow\infty$, uniformly for $\lambda \in B(1,\delta)$, we deduce by the Koebe quarter theorem that 
\begin{equation}
\label{balliscontained}
B(0,10) \subset f_\lambda(B(\xi_m, r)),
\end{equation}
  for all sufficiently large $m$, whenever $\lambda\in B(1, \delta)$.

Suppose that $n$ is large. Observe that the mapping $\psi_1$ which takes $\lambda$ to $f_\lambda(z_n)$ maps $B(1, \delta)$ to $B(2+z_n, \delta|2+z_n|)$. Hence, if $m$ is sufficiently large, then $n$ can be chosen such that $ B(\xi_m, r) \subset \psi_1(B(1, \delta)),$ and also $|\psi_1(\lambda) - \xi_m| \geq 2r$ whenever $\lambda \in \partial B(1, \delta)$. Let $\alpha_\lambda$ be the branch of $f_\lambda^{-1}$ which maps $B(0,10)$ to $B(\xi_m, r)$.

Suppose that $\zeta \in B(0, 10)$, and let $\alpha\colon B(1, \delta)\to B(\xi_m, r); \lambda\mapsto \alpha_\lambda(\zeta)$. We deduce by 
   Rouch\'{e}'s theorem that there exists $\lambda \in B(1, \delta)$ such that $\psi_1(\lambda) - \alpha(\lambda)=0$, which is equivalent to $f_\lambda^2(z_n) = \zeta$.

Now consider the mapping $\psi_2\colon B(1, \delta) \to \mathbb{C}; \lambda\mapsto f_{\lambda}^2(z_n)$. Provided that $n$ is sufficiently large, 
   it follows by the argument above that $B(0, 10) \subset \psi_2(B(1, \delta))$.

Let $K \subset B(1, \delta)$ be a component of $\psi_2^{-1}(B(0, 10))$. Then $$|\phi(\lambda)| \leq 2\pi < 10 = |\psi_2(\lambda)|,$$ whenever $\lambda \in \partial K$. Hence we may apply Rouch{\'e}'s theorem again, and deduce that there exists $\lambda_0 \in B(1, \delta)$ such that $\psi_2(\lambda_0) - \phi(\lambda_0) = 0$. In other words $f^2_{\lambda_0}(z_n)$ is a repelling fixed point of $f_{\lambda_0}$ and so lies in $J(f_{\lambda_0})$. This completes the proof of our claim.
       \end{proof} 

\newcommand{\etalchar}[1]{$^{#1}$}
\def\cprime{$'$}
\providecommand{\bysame}{\leavevmode\hbox to3em{\hrulefill}\thinspace}

\end{document}